\documentclass[reqno]{amsart}

\usepackage{amsmath,amsthm,amscd,amssymb,amsfonts, amsbsy}

\usepackage[usenames, dvipsnames]{color}
\usepackage{enumitem}
\usepackage{mathrsfs}
\usepackage{hyperref}
\usepackage{esint}
\usepackage{verbatim}

\theoremstyle{plain}
\newtheorem{theorem}{Theorem}[section]
\newtheorem{lemma}[theorem]{Lemma}
\newtheorem{corollary}{Corollary}[section]

\theoremstyle{definition}

\theoremstyle{remark}
\newtheorem{remark}[theorem]{Remark}

\numberwithin{equation}{section}

\makeatletter
\@namedef{subjclassname@2020}{%
  \textup{2020} Mathematics Subject Classification}
\makeatother

\newcommand{\bR}{\mathbb{R}}

\newcommand\sL{\mathscr{L}}

\newcommand\oE{\mathsf{E}}

\begin{document}
\title[Trace theorems for weighted mixed norm Sobolev spaces]{On trace Theorems for weighted mixed norm Sobolev spaces and applications}

\author{Tuoc Phan}
\address{Department of Mathematics, University of Tennessee, Knoxville, 227 Ayres Hall, 1403 Circle Drive, Knoxville, TN 37996, U.S.A.}
\email{phan@utk.edu}
\subjclass[2020]{46E35, 35J25, 35J70, 35J75.}
\keywords{Trace theorems, Weighted mixed-norm Sobolev spaces, Mixed-norm Besov spaces, Boundary value problems, Singular and degenerate coefficients.}

\begin{abstract} We prove trace theorems for weighted mixed norm Sobolev spaces in the upper-half space where the weight is a power function of the vertical variable. The results show the differentiability order of  the trace functions depends only on the power in the weight function and the integrability power for the integration with respect to the vertical variable but not on the integrability powers for the integration with respect to the horizontal ones. They are new even in the  un-weighted case and they recover classical results in the case of un-mixed norm spaces. The work is motivated by the study of regularity theory for solutions of elliptic and parabolic equations with anisotropic features and with non-homogeneous boundary conditions. The results provide an essential ingredient to the study of fractional elliptic and parabolic equations in divergence form with measurable coefficients.
\end{abstract}
\maketitle
\section{introduction and main results.}
\subsection{Introduction} We study traces of functions in weighted mixed-norm Sobolev spaces. To put the study into perspectives, let us quickly recall a relevant classical result. For a given open non-empty bounded domain $\Omega \subset \bR^d$ with $d \in \mathbb{N}$, let $Q= \Omega \times (0,1)$.  For $p \in [1, \infty)$ and $\alpha \in \mathbb{R}$, we denote  $L_{p}(Q, \mu)$  the weighted Lebesgue space consisting of measurable functions $f: Q \rightarrow \bR$ such that  
\begin{equation*} 
\|f\|_{L_{p}(Q,\mu)} = \left[ \int_0^1  \int_{\Omega} |f(x, y)|^p y^\alpha dx dy \right]^{\frac{1}{p}}  <\infty
\end{equation*}
where $\mu(y) = y^\alpha$ for $y \in \bR_+$. As usual,  the weighted Sobolev space $W^{1}_{p}(Q, \mu)$ is defined by
\[
W^{1}_{p}(Q, \mu) =\Big \{ u \in L_p(Q, \mu): Du \in L_p(Q, \mu)\Big\}, 
\]
where  $Du = (D_x u, D_yu)$ denotes the gradient of $u$ in the weak sense, and $W^{1}_{p}(Q, \mu) $ is endowed with the norm
\[
\|u\|_{W^{1}_{p}(Q, \mu)} = \|u\|_{L_{p}(Q,\mu)} + \|Du\|_{L_{p}(Q, \mu)}.
\]
On the other hand, for  $s \in (0,1)$ and $p \in [1, \infty)$, a function $v \in L_p(\Omega)$ is said in the Sobolev space $W^{s}_{p}(\Omega)$ if 
\[
\|v\|_{W^{s}_{p}(\Omega)} = \| v\|_{L^p(\Omega)} +\left( \int_{\Omega} \int_{\Omega} \frac{|v(x) - v(z)|^p}{|x-z|^{d + ps}} dx dz \right)^{\frac{1}{p}}< \infty.
\]
The following trace theorem for weighted Sobolev space $W^{1}_{p}(Q, \mu)$ is classical. See \cite[Theorem 2.8]{Nekvinda}, for example. See also \cite[Theorem 2.10] {DK} for a more general domains and general cases, and  \cite[Theorem 2.8, p. 100]{Necas} for the un-weighted case, i.e. $\alpha=0$.
\begin{theorem} \label{Nekvinda-thm} Let $\Omega \subset \bR^d$ be a bounded non-empty set with boundary $\partial \Omega \in C^1$. Also let $p \in [1, \infty)$, and $\alpha \in (-1, p-1)$. Then, there exists a unique linear trace operator 
\[
\mathsf{T}: W^{1}_{p}(Q, \mu) \rightarrow W^{1- \frac{1+\alpha}{p}}_p(\Omega)
\]
such that $\mathsf{T}u(\cdot) = u (\cdot, 0)$ in $\Omega$ for $u \in C(\Omega \times [0, 1)) \cap W^{1}_{p}(Q, \mu)$. Moreover, we have
\[
\|\mathsf{T} u\|_{W^{1- \frac{1+\alpha}{p}}_p(\Omega)} \leq N \|u\|_{W^{1}_{p}(Q, \mu)},
\]
for every $u \in W^{1}_{p}(Q, \mu)$, where $\mu(y) = y^\alpha$, $N = N(\alpha, d, p) >0$, and $Q = \Omega \times (0,1)$.
\end{theorem}

The main theme of this paper is to investigate the traces of functions in weighted  mixed-norm spaces $W^1_{\vec{p}, q}(\bR^{d+1}_+, \mu)$ and $W^1_{p, q}(Q, \mu)$ which are defined below, where $\vec{p} \in [1, \infty)^d$ and  $p$ are the  powers in the integrability with respect to the horizontal variables, and $q \in [1, \infty)$ is the power in the integrability with respect to the vertical variable. One obvious motivation is to understand the analysis properties of weighted mixed-norm Sobolev spaces and extend the classical result, Theorem \ref{Nekvinda-thm}, to the anisotropic setting. Another  and more important motivation is to provide an essential ingredient for the study of partial differential equations with  anisotropic structures and with non-homogeneous boundary conditions. 

\smallskip
Regarding the partial differential equations with  anisotropic features, one example is a class of elliptic and parabolic equations studied in \cite{DP21, DP20, DP19}  in which coefficients are singular or degenerate in one variable direction as in \eqref{expl-PDE} below. Existence, uniqueness, and regularity estimates of solutions in weighted mixed-norm spaces are proved in \cite{DP21, DP20, DP19, MNS, STV} for the problems with homogeneous boundary conditions. Using the developed trace theorem (Theorem \ref{thrm1} below), we extend \cite[Theorem 2.8]{DP21}  to the study of non-homogeneous boundary value problem \eqref{expl-PDE}, and our result (Corollary \ref{thrm3} below) gives the optimal regularity conditions on the boundary data for the problem. In this research line, interested readers can find \cite{Lin, Zhang} for other results related to equations with singular-degenerate coefficients, and also \cite{KKL, Nazarov, Krylov} for interesting results on elliptic and parabolic equations in weighted spaces similar to the space $W^1_{\vec{p}, q}(\bR^{d+1}_+, \mu)$ considered in this paper.

\smallskip
Next, we give another example about the partial differential equations with  anisotropic structures. In \cite{MP}, a class fractional elliptic equations with measurable coefficients in divergence form in bounded domain was studied. The main idea in \cite{MP} was to convert the non-local problem into the local on using the extension operator introduced in \cite{Caff-Sil}. To this end, \cite{MP} developed regularity theory for a class of equations with the anisotropic structures as coefficients are singular or degenerate in one variable direction. In one of its main results,  \cite[Theorem 2.1]{MP},  the trace theorem \ref{Nekvinda-thm} plays a central role in the establishment of the regularity estimates of solutions. Because of this, and due to the anisotropic properties due to the extension problem, it is clearly expected that \cite[Theorem 2.1]{MP} is not optimal. In this paper, we establish the trace theorem for the anisotropic weighted Sobolev space $W^{1}_{p, q}(Q, \mu)$ which will provide us an essential ingredient to derive the optimal regularity result for solutions of the class of fractional elliptic equations studied in \cite{MP}. As this application is comprehensive, the work will be carried out in our forthcoming paper.  See also \cite{CT} for another interesting work in which Theorem \ref{Nekvinda-thm} is useful, and also \cite{Caff-Sil} for well-known results on equations with fractional Laplacian. 

\smallskip
Finally, we would like to point out that the interest in this work is also closely related to the recent work \cite{BGT, JHS} in which the traces of functions in un-weighted mixed-norm Lizorkin-Triebel spaces were studied. See also \cite{AKS, T} for other results on traces of functions in Sobolev spaces with radial weights and with Muckenhoupt weights. Also, note that from Theorem \ref{thrm1} and Corollary \ref{thrm2}, the differentiability order of the traces of functions in $W_{\vec{p}, q}^1(\bR^{d+1}_+, \mu)$ or in $W^1_{p, q}(Q, \mu)$ only depends on the weight power $\alpha$ and the integrability power $q$ of the vertical variable,  but not on horizontal integrability powers $\vec{p}$ or $p$, respectively. This phenomena is not easily recognized in Theorem \ref{Nekvinda-thm} and to the best of our knowledge, it seems to be newly discovered in this paper, even in the un-weighted case with $\alpha =0$. See Remark \ref{remark-1-intro} below for more details.

\subsection{Main results} Let us begin with recalling the definitions of the mixed-norm and weighted mixed-norm spaces.  For each $\vec{p} = (p_1, p_2,\ldots, p_d) \in [1, \infty)^d$, the mixed-norm Lebesgue space $L_{\vec{p}}(\bR^d)$ is defined to be the space consisting of all measurable functions $f: \bR^d \rightarrow \bR$ such that its $L_{\vec{p}}(\bR^d)$-norm
\[
\|f\|_{L_{\vec{p}}(\bR^d)} = \left(\int_{\bR} \ldots \left (\int_{\bR}\left( \int_{\bR} |f(x_1, x_2, \ldots, x_d)|^{p_1} dx_1 \right)^{\frac{p_2}{p_1}} dx_2 \right)^{\frac{p_3}{p_2}} \ldots dx_d \right)^{\frac{1}{p_d}}
\]
is finite. Similar definitions can be formed when $p_k =\infty$ for some $k =1, 2,\ldots, d$. The mixed-norm Lebesgue space $L_{\vec{p}}(\bR^d)$ was introduced in \cite{BP} to capture the anisotropic behaviors of functions. Clearly, if $p_1 = p_2 = \ldots = p_d =p$, then $L_{\vec{p}}(\bR^d) = L_p(\bR^d)$, where the later is the usual Lebesgue space. However, in general, there is not a clear relation between the two functional spaces. For example, for a given $\vec{p} = (p_1, p_2, \ldots, p_d) \in [1, \infty)^d$, and for $f_k \in L_{p_k}(\bR)$ for $k = 1, 2, \ldots, d$, we observe that 
$f \in L_{\vec{p}}(\bR^d)$, where
\[
f(x) = f_1(x_1) f_2(x_2) \ldots f_d(x_d), \quad x = (x_1, x_2,\ldots, x_d) \in \bR^d.
\]
However, there may not exist any $p \in [1, \infty)$ such that $f \in L_{p}(\bR^d)$. 

\smallskip
Next,  let us denote $\bR_+ = (0, \infty)$ and $\bR^{d+1}_+ = \bR^{d} \times \bR_+$. For $\vec{p} \in [1, \infty]^d, \alpha \in \bR$, and $q \in [1,\infty)$, the weighted mixed-norm  Lebesgue space $L_{\vec{p}, q}(\bR^{d+1}_+, \mu) $ is defined by 
\[
L_{\vec{p}, q}(\bR^{d+1}_+, \mu) = \{f: \bR^{d+1}_+ \rightarrow \bR: \|f\|_{L_{\vec{p}, q}(\bR^{d+1}_+)} <\infty\}
\]
where $\mu(y) = y^\alpha$ for $y \in \bR_+$ and 
\[
\|f\|_{L_{\vec{p},q}(\bR^{d+1}_+, \mu)} = \left( \int_{\bR_+}\|f(\cdot, y)\|_{L_{\vec{p}}(\bR^d)} ^q y^\alpha dy \right)^{1/q}.
\]
A similar definition is easily formulated when $q =\infty$. As usual, we define the weighted mixed-norm Sobolev space $W^1_{\vec{p}, q}(\bR^{d+1}_+, \mu)$ to be the space consisting of functions $f \in L_{\vec{p}, q}(\bR^{d+1}_+, \mu)$ such that their weak derivatives $Df  = (D_x f, D_y f) \in L_{\vec{p}, q}(\bR^{d+1}_+, \mu)^{d+1}$, and it is endowed with the norm
\[
\|f\|_{W^1_{\vec{p}, q}(\bR^{d+1}_+, \mu)} = \|f\|_{L_{\vec{p},q}(\bR^{d+1}_+, \mu)} + \|Df\|_{L_{\vec{p},q}(\bR^{d+1}_+, \mu)}.
\]

\smallskip
Now, we recall the definition of mixed-norm Besov spaces. For each $h \in \mathbb{R}^d$ and for $f: \bR^d \rightarrow \bR$, let us define the difference $\Delta_h f$ by
\[
\Delta_h f(x) = f(x +h) - f(x), \quad x \in \mathbb{R}^d.
\]
For $\vec{p} \in [1, \infty]^d,  q \in [1, \infty)$ and $\ell \in (0,1)$, the mixed-norm Besov space $B^{\ell}_{\vec{p}, q}(\bR^d)$ is the space consisting of all $f \in L_{\vec{p}}(\bR^d)$ such that its norm
\[
\|f\|_{B^{\ell}_{\vec{p}, q}(\bR^d)} =\|f\|_{L_{\vec{p}}(\bR^d)} + \|f\|_{b^{\ell}_{\vec{p}, q}(\bR^d)} <\infty,
\]
where
\[
\|f\|_{b^{\ell}_{\vec{p},q}(\bR^d)} =\left[ \int_{\bR^d} \Big( \frac{\|\Delta_h f\|_{L_{\vec{p}}(\mathbb{R}^d)}}{|h|^\ell} \Big)^q \frac{d h}{|h|^d} \right]^{\frac{1}{q}}.
\]
Observe that when $\vec{p} = (p, p, \ldots, p)$, $B^{\ell}_{\vec{p}, q}(\bR^d) = B^{\ell}_{p, q}(\bR^d)$, where the later is the usual Besov space.

\smallskip
The following trace theorem is the main result of the paper.
\begin{theorem}[Trace theorem] \label{thrm1} Let $\vec{p} \in [1, \infty)^d$, $q \in [1, \infty)$, $\alpha \in (-1, q-1)$, and $\ell = 1 - \frac{1+\alpha}{q}$. Then, there exists $N = N(d, \vec{p}, q, \alpha)>0$ such that the following assertions hold true.
\begin{itemize}
\item[\textup{(i)}]  There is a trace map $\mathsf{T}: u \in W_{\vec{p}, q}^1(\bR^{d+1}_+ , \mu)  \rightarrow  B^{\ell}_{\vec{p}, q}(\bR^d)$ such that
\begin{equation} \label{trace-est}
\|\mathsf{T}(u)\|_{B^{\ell}_{\vec{p}, q}(\bR^d)} \leq N \|u\|_{W_{\vec{p}, q}^1(\bR^{d+1}_+,\mu)}, \quad \forall \ u \in W_{\vec{p}, q}^1(\bR^{d+1}_+,\mu)
\end{equation}
and $\mathsf{T}(u)(\cdot, 0) = u(\cdot,0)$ if $u \in C(\overline{\bR}^{d+1}_+) \cap  W_{\vec{p}, q}^1(\bR^{d+1}_+,\mu)$.
\item[\textup{(ii)}] There exists an extension map $\oE: B^{\ell}_{\vec{p}, q}(\bR^d) \rightarrow W_{\vec{p}, q}^1(\bR^{d+1}_+, \mu)$ satisfying
\[
\|\oE(g)\|_{W_{\vec{p}, q}^1(\bR^{d+1}_+, \mu)} \leq N \|g\|_{B^{\ell}_{\vec{p}, q}(\bR^d)}, \quad \forall~g \in B^{\ell}_{\vec{p}, q}(\bR^d).
\]
Moreover, for every $g \in B^{\ell}_{\vec{p}, q}(\bR^d)$, $\textup{spt}(\oE(g)(x, \cdot)) \subset [0, 1)$ for all $x \in \bR^d$ and 
\[
\lim_{y\rightarrow 0^+} y^{-\ell} \|\oE(g)(\cdot, y) - g\|_{L_{\vec{p}}(\bR^d)} =0.
\]
\end{itemize}
\end{theorem}
\begin{remark} \label{remark-1-intro} As $\ell$ only depends on $\alpha$ and $q$, but not on $\vec{p}$, we see that the differentiability order of the traces of functions in $W_{\vec{p}, q}^1(\bR^{d+1}_+, \mu)$ only depends on $\alpha$ and $q$. This phenomena does not seem to be seen before even in the unweighted case (i.e. when $\alpha =0$) and with $d=1$. When $\vec{p} = (q, q, \ldots, q)$, Theorem \ref{thrm1} is reduced to the classical results, see \cite[Theorem 2.8]{Nekvinda}, \cite[Theorem 2.10] {DK}, and \cite[Theorem 2.8, p. 100]{Necas}. We also remark that when $\alpha \leq -1$, then it follows from Lemma \ref{L-p-trace} below that $\mathsf{T}(u) \equiv 0$ for any $u \in W_{\vec{p}, q}^1(\bR^{d+1}_+,\mu)$ for any $\vec{p} \in [1, \infty)^d$ and $q \in [1, \infty)$.  
\end{remark}
\smallskip
We next give two applications of Theorem \ref{thrm1} which also demonstrates some motivations for this study.  In the first application, we study a class of second order elliptic equations in divergence form with coefficients that can be singular or degenerate near the boundary of the domains.  Let $(a_{ij}): \bR^{d+1}_+ \rightarrow \bR^{(d+1)^2}$ be the $(d+1) \times (d+1)$ matrix of measurable functions satisfying the uniformly elliptic and boundedness condition: there is $\nu \in (0,1)$ such that
\begin{equation} \label{ellipticity}
a_{ij}(x,y) \xi_i \xi_j \geq \nu |\xi|^2 \quad \text{and} \quad |a_{ij}(x,y)| \leq \nu^{-1}, 
\end{equation}
for all $(x,y) \in \bR^{d+1}_+$ and for all $\xi =(\xi_1, \xi_2,\cdots, \xi_{d+1}) \in \bR^{d+1}$. For each $\lambda >0$, let us denote the operator  
\[
\sL u(x,y) = - D_i\big(y^\beta a_{ij}(x,y) D_j u(x,y) \big) + \lambda y^\beta u(x,y) , \]
where $\beta \in (-\infty,1)$ is a fixed number, and 
\[
D_i u (x,y) = \frac{\partial u}{\partial x_i}(x,y) \quad \text{for} \quad i =1, 2, \ldots, d \quad \text{and} \quad D_{d+1} u(x,y) = \frac{\partial u}{\partial y}(x,y)
\]
for $(x,y) \in \bR^d \times \bR_+$. Note that  in the above the Enstein's summation convention is used. 

\smallskip
As $\beta \in (-\infty, 1)$, the leading coefficients $y^\beta a_{ij}(x,y)$ in $\sL$ can be singular or degenerate on the boundary $\partial \bR^{d+1}_+$ of  $\bR^{d+1}_+$. When $\beta =0$, the operator $\sL$ is reduced to the standard second order partial differential operator with bounded and uniformly elliptic coefficients. When $\beta \in (-1,1)$ and $(a_{ij})$ is the identity matrix, the operator $\sL$ appears in the study of fractional Laplace equations, see \cite{Caff-Sil} for instance. Interested readers may find in \cite{Lin, Zhang},  for instance, for some other work related to problems in geometric and probability in which equations with singular and degenerate coefficients also appear. 

\smallskip
We study the non-homogeneous boundary value problem
\begin{equation} \label{expl-PDE}
\left\{
\begin{array}{cccl}
\sL u(x,y)  & = & D_i (y^\beta F_i) + \sqrt{\lambda} y^\beta f & \quad (x, y) \in  \bR^{d} \times \bR_+\\
u (x, 0) & =& g(x) & \quad x \in  \bR^{d},
\end{array} \right. 
\end{equation}
where  $F_i : \bR^{d+1}_+ \rightarrow \bR$ is a given measurable function for $i =1, 2,\ldots, d+1$, and $f : \bR^{d+1}_+ \rightarrow \bR$ is a given measurable function. To state our result for \eqref{expl-PDE}, we need to impose a weighted partial VMO condition on the coefficients $(a_{ij})$. For every $r>0$ and $x_0 \in \bR^d, y_0 \in [0, \infty)$, we denote
\[
D_r(x_0, y_0) = \big (B_r(x_0) \times (y_0 -r, y_0 +r) \big)\cap \bR^{d+1}_+
\]
where $B_r(x_0)$ is the ball in $\bR^d$ of radius $r$ and centered at $x_0$. We also denote the  measure $\mu_1(dxdy) = y^{-\beta} dx dy$ and
\[
  \fint_{D_r(x_0,y_0)} f(x,y) \mu_1(dxdy)= \frac{1}{\mu_1(D_r(x_0,y_0))} \int_{D_{r(x_0,y_0)}} f(x,y) \mu_1(dxdy)
\]
where
\[
\mu_1(D_r(x_0,y_0)) = \int_{D_r(x_0,y_0)} \mu_1(dxdy).
\]
 Also, we denote the average of $a_{ij}$ in $B_r(x_0)$ by
\[
[a_{ij}]_{r, x_0}(y) = \frac{1}{|B_r(x_0)|}\int_{B_r(x_0)} a_{ij}(x, y) dx, \quad y \in \bR_+.
\]
Then, the partially bounded mean oscillations of the matrix $a = (a_{ij})$ with respect to the weight $\mu_1$ in $D_r(x_0, y_0)$ is defined by
\[
a^{\#}_r(x_0,y_0) = \max_{i,j =1,2,\ldots, d+1} \fint_{D_{r}(x_0, y_0)}|a_{ij}(x,y) -[a_{ij}]_{r, x_0}(y)| \mu_1(dxdy).
\]
Recall that $\mu(dxdy) = y^{\alpha} dx dy$. For every $\vec{p} \in (1, \infty)^d$, $q \in (1, \infty)$, $\alpha \in (-1, q-1)$, and with $g \in B_{\vec{p},q}^{\ell}(\bR^d)$ with $\ell = 1-\frac{1+\alpha}{q}$, we say that $u \in W^{1}_{\vec{p}, q}(\bR^{d+1}, \mu)$ is a weak solution to \eqref{expl-PDE} if
\[\int_{\bR^{d+1}_+} y^{\beta}\big[a_{ij} D_j u D_j \varphi + F_i D_i\varphi + \lambda u \varphi\big] dxdy = \sqrt{\lambda} \int_{\bR^{d+1}_+} f(x,y) y^\beta dxdy\]
for all smooth, compactly supported function $\varphi: \bR^{d+1}_+ \rightarrow \bR$,  and $u (x,0) = g(x)$ in sense of trace. Now, our result regarding \eqref{expl-PDE} is stated in the following theorem.
\begin{corollary} \label{thrm3} Let $\vec{p} \in (1, \infty)^d$, $q \in (1, \infty)$, $\beta \in (-\infty, 1)$,  $\alpha \in (-1, q-1)$,  $\ell = 1- \frac{1+\alpha}{q}$, $\nu \in (0,1)$, and $R_0 \in (0, \infty)$. There exist sufficiently small positive number $\delta~=~ \delta(d, \vec{p}, q, \beta, \alpha, \nu)$ and a sufficiently large positive number $\lambda_0~=~\lambda_0(d, \vec{p}, q, \beta, \alpha, \nu)$ such that the following assertions hold. Suppose that \eqref{ellipticity} and
\begin{equation} \label{VMO-cond}
\sup_{(x_0, y_0) \in \overline{\bR}^{d+1}_+}\sup_{r\in (0, R_0)} a^{\#}_r(x_0,y_0) \leq \delta
\end{equation}
hold. If $u \in W^{1}_{\vec{p}, q}(\bR^{d+1}, \mu)$ is a weak solution of \eqref{expl-PDE} with $f \in L_{\vec{p},q}(\bR^{d+1}_+, \mu), F \in L_{\vec{p},q}(\bR^{d+1}_+, \mu)^{d+1}$, $g \in B^{\ell}_{\vec{p}, q}(\bR^d)$, and $\lambda\geq \lambda_0 R_0^{-2}$, then
\[
\begin{split}
& \|Du\|_{L_{\vec{p}, q}(\bR^{d+1}_+, \mu)} + \sqrt{\lambda} \|u\|_{L_{\vec{p}, q}(\bR^{d+1}_+, \mu)} \\
&\leq N\Big[ \|F\|_{L_{\vec{p}, q}(\bR^{d+1}_+, \mu)} + \|f\|_{L_{\vec{p}, q}(\bR^{d+1}_+, \mu)}  \Big] \\
& \qquad + N\Big[\|D\oE(g)\|_{L_{\vec{p}, q}(\bR^{d+1}_+, \mu)} + \sqrt{\lambda} \|\oE(g)\|_{L_{\vec{p}, q}(\bR^{d+1}_+, \mu)} \Big],
\end{split}
\]
where $N = N(d, \vec{p}, q, \beta, \alpha, \nu) >0$ and $\oE$ is defined in Theorem \ref{thrm1}. Moreover, for every $f \in L_{\vec{p},q}(\bR^{d+1}_+, \mu), F \in L_{\vec{p},q}(\bR^{d+1}_+, \mu)^{d+1}$, $g \in B^{\ell}_{\vec{p}, q}(\bR^d)$, and $\lambda \geq \lambda_0 R_0^{-2}$, there exists a unique weak solution $u \in W^{1}_{\vec{p}, q}(\bR^{d+1}, \mu)$ to \eqref{expl-PDE}.
\end{corollary}
We remark that the problem \eqref{expl-PDE} when $g \equiv 0$ is studied in \cite{DP21} for both elliptic and parabolic cases. Corollary \ref{thrm3} therefore recovers \cite[Theorem 2.8]{DP21} for  \eqref{expl-PDE} when $g\equiv 0$. Even when $\vec{p} = (p_1, p_2, \ldots, p_d)$ with $p_1= p_2 = \ldots = p_d$ and $\beta =0$, Corollary \ref{thrm3} seems to be new as it deals with non-homogeneous boundary value problems in weighted mixed-norm spaces. We also point out that when $\alpha \leq -1$, to search for a solution in $W^{1}_{\vec{p}, q}(\bR^{d+1}_+, \mu)$ of \eqref{expl-PDE}, it requires that $g \equiv 0$ as pointed out in Remark \ref{remark-1-intro}. In fact, when $\alpha \in (q\beta - 1, -1]$ and $g \equiv 0$, the existence and uniqueness of weak solutions $W^{1}_{\vec{p}, q}(\bR^{d+1}_+, \mu)$ can be obtained from \cite{DP21}.

\smallskip
In the second application of Theorem \ref{thrm1}, we prove the trace theorem for $W^1_{p,q}(\Omega \times (0,1), \mu)$ where $\Omega \subset \bR^d$ is open and bounded domain with Lipschitz boundary $\partial \Omega$. As before, for $p, q \in [1, \infty)$ and $Q=\Omega \times (0,1)$, let us denote $L_{p, q}(Q, \mu)$ be the weighed mixed normed Lebesgue space consisting of measurable functions $f: Q \rightarrow \mathbb{R}$ such that
\[
\|f\|_{L_{p,q}(Q, \mu)} = \left( \int_0^1 \left( \int_{\Omega} |f(x, y)|^p dx\right)^{\frac{q}{p}} y^\alpha dy \right)^{1/q} <\infty.
\]
Similar to the weighted mixed-norm Sobolev space $W^1_{\vec{p}, q}(\bR^{d+1}_+, \mu)$, let us denote the weighted mixed-norm Sobolev space $W^1_{p,q}(Q, \mu)$ by
\[
W^1_{p,q}(Q, \mu) = \Big \{f \in L_{p,q}(Q, \mu): Df \in L_{p,q}(Q, \mu) \Big\}
\]
and it is endowed with the norm
\[
\|f\|_{W^1_{p,q}(Q, \mu)} = \|f\|_{L_{p,q}(Q, \mu)} + \|Df\|_{L_{p,q}(Q, \mu)}.
\]
For $\ell \in (0,1)$, and $p, q \in [1, \infty)$, we define the Besov space $B^{\ell}_{p,q}(\Omega)$ as
\[
B^{\ell}_{p,q}(\Omega) =\Big\{ f \in L_p(\Omega): \exists\ g \in B^{\ell}_{p,q}(\bR^d): g_{|\Omega} = f   \Big\}
\]
and it is endowed with the norm
\[
\|f\|_{B^{\ell}_{p,q}(\Omega)} = \inf\Big\{\|g\|_{B^{\ell}_{p,q}(\bR^d)} \ \text{for} \ g \in B^{\ell}_{p,q}(\bR^d): g_{|\Omega} = f   \Big\}.
\]
Interested readers can find  \cite{DS, VR, Trib1} and the monograph \cite{Trib} for more details on Besov spaces in bounded domains. Note  that it is possible to use the vector $\vec{p}$ instead of the scaler $p$ in the definition $B^{\ell}_{p,q}(\Omega)$, however we do not pursue this to avoid complication in writing the mixed-norm of functions in $L_{\vec{p}}(\Omega)$  due to the geometry of $\Omega$. 

\smallskip
Our trace theorem for weighted mixed-norm space $W^1_{p, q}(Q, \mu)$ with $Q=  \Omega \times (0,1)$ and $\mu(y) = y^\alpha$ for $y \in (0,1)$ is now stated below.
\begin{corollary} \label{thrm2} Assume $\Omega \subset \bR^d$ is an open bounded domain with Lipschitz boundary $\partial \Omega$ and $Q = \Omega \times (0,1)$. Assume also that $p, q \in [1, \infty)$, $\alpha \in (-1, q-1)$ and $\ell = 1 -\frac{1+\alpha}{q}$.  There exists a linear operator 
$$\mathsf{T}_{\Omega}: u \in W_{p, q}^1(Q , \mu)  \rightarrow  B^{\ell}_{p, q}(\Omega)$$
such that  $\mathsf{T}_{\Omega}(u)(\cdot, 0) = u(\cdot,0)$ if $u \in C(\Omega \times [0, 1)) \cap  W_{p, q}^1(Q,\mu)$ and
\begin{equation*} 
\|\mathsf{T}_{\Omega}(u)\|_{B^{\ell}_{p, q}(\Omega)} \leq N \|u\|_{W_{p, q}^1(Q,\mu)}, \quad \forall \ u \in W_{p, q}^1(Q,\mu),
\end{equation*}
where $N = N(p, q, d, \alpha, \Omega)>0$ and $\mu(y) = y^\alpha$ for $y \in (0,1)$.
\end{corollary}
As previously mentioned, Corollary \ref{thrm2} provides key ingredient for the study of the regularity in Besov spaces of solutions to the class of non-local elliptic equations with measurable coefficients and this will be carried out in our forthcoming paper. When $p =q$, Corollary \ref{thrm2} is reduced to the classical result, Theorem \ref{Nekvinda-thm}. As in Theorem \ref{thrm1}, it is interesting to note that $\ell$ only depends on $\alpha, q$ but not on $p$, the phenomena is not implied by Theorem \ref{Nekvinda-thm}. It is possible to extend Theorem \ref{thrm1} and Corollary \ref{thrm2} to higher-ordered $W^k_{p, q}$-weighted mixed-norm Sobolev spaces with $k \in \mathbb{N}$ and more general bounded domains as in \cite[Theorem 2.10] {DK}. However, we do no pursue this direction to avoid the technical complexity in this paper. 

\smallskip
The remaining parts of the paper is organized as follows. In Section \ref{preli-sess}, several preliminary estimates and inequalities needed to prove  Theorem \ref{thrm1} are recalled and proved. Section \ref{Sec-2} is devoted to the proof of Theorem \ref{thrm1}. The proofs of Corollary \ref{thrm3} and Corollary \ref{thrm2} will be given in Section \ref{Sec-3}.

\noindent
\section{Preliminary estimates and inequalities} \label{preli-sess}
\subsection{Hardy's inequalities}
We begin with the following weighted Hardy's inequality which is useful in the paper.
\begin{lemma}[Hardy's inequality] \label{Hardy-ineq} Let $q \in [1, \infty)$, $a \in (0, \infty]$, and $\sigma < 1 - \frac{1}{q}$. Then, we have
\[
\left( \int_0^a t^{q \sigma} \Big | \frac{1}{t} \int_0^t f(s) ds \Big|^q dt \right)^{\frac{1}{q}} \leq \Big(1 - \frac{1}{q} -\sigma\Big)^{-1} \left( \int_0^a |f(t)|^q  t^{q\sigma} dt \right)^{\frac{1}{q}},
\]
for every measurable function $f: (0, a) \rightarrow \mathbb{R}$. Similar statement also holds when $q =\infty$.
\end{lemma}
\begin{proof} The proof is standard and we give it here for completeness. We only consider the case $q<\infty$. Note that with the change of variable $ s = ty$, we have
\[
\left( \int_0^a t^{q \sigma} \Big | \frac{1}{t} \int_0^t f(s) ds \Big|^q dt \right)^{\frac{1}{q}} \leq \left( \int_0^a t^{q \sigma} \Big( \int_0^1 |f(t y)| dy \Big)^q dt \right)^{\frac{1}{q}}.
\]
Then, it follows from Minkowski's inequality that
\[
\begin{split}
\left( \int_0^a t^{q \sigma} \Big | \frac{1}{t} \int_0^t f(s) ds \Big|^q dt \right)^{\frac{1}{q}} & \leq \int_0^1  \Big (  \int_0^a |f(t y)|^q t^{q \sigma}dt \Big)^{\frac{1}{q}} dy \\
& = \int_0^1  y^{-\sigma -\frac{1}{q}} dy \left(  \int_0^{ay} |f(\tau)|^q \tau^{q \sigma}d\tau \right)^{\frac{1}{q}} \\
& \leq \left(  \int_0^{a} |f(\tau)|^q \tau^{q \sigma}d\tau \right)^{\frac{1}{q}} \int_0^1  y^{-\sigma -\frac{1}{q}} dy \\
& =\Big(1 - \frac{1}{q} -\sigma \Big)^{-1} \left(  \int_0^a |f(\tau)|^q \tau^{q \sigma}d\tau \right)^{\frac{1}{q}},
\end{split}
\]
where we have used the change of variable $ \tau = ty$ in the second step of the above estimate. The assertion is proved.
\end{proof}
The following consequence of Lemma \ref{Hardy-ineq} is needed in this paper.
\begin{corollary} \label{Hardy-cor} Let $\theta \in [1, \infty), a \in (0, \infty]$ and $\beta < d -\frac{1}{\theta}$. Then, there is a constant $N = N(d, \beta, \theta)>0$ such that
\[
\left(\int_0^a \Big( t^{\beta -d} \int_{|x| \leq t} |f(x)| dx \Big)^\theta dt \right)^{1/\theta} \leq N \left(\int_{|x| \leq a}\Big( |x|^{\beta -\frac{d-1}{\theta}}|f(x)|\Big)^\theta dx \right)^{1/\theta},
\]
for every measurable function $f: \mathbb{R}^d \rightarrow \mathbb{R}$. Similar statement also holds when $\theta = \infty$.
\end{corollary} 
\begin{proof} We only provide the proof when $\theta <\infty$. Using the polar coordinate and Minkowski's inequality, we have
\[
\begin{split}
& \left(\int_0^a \Big( t^{\beta -d} \int_{|\eta| \leq t} |f(x)| dx \Big)^\theta dt \right)^{1/\theta} \\
& =\left(\int_0^a \Big( \int_{\mathbb{S}^{d-1}}  t^{\beta -d} \int_{0}^{t} |f(r \xi)| r^{d-1} dr d\xi \Big)^\theta dt \right)^{1/\theta} \\
& \leq \int_{\mathbb{S}^{d-1}}  \left( \int_0^a t^{(\beta -d +1) \theta} \Big| \frac{1}{t} \int_{0}^{t} |f(r \xi)| r^{d-1} dr  \Big|^\theta dt \right)^{1/\theta} d\xi.
\end{split}
\]
Then, applying Lemma \ref{Hardy-ineq} for $q = \theta$ and $\sigma = \beta -d +1$, we obtain
\[
\begin{split}
& \left(\int_0^a \Big( t^{\beta -d} \int_{|x| \leq t} |f(x)| dx \Big)^\theta dt \right)^{1/\theta} \\
& \leq N \left(\int_{\mathbb{S}^{d-1}}\int_0^a r^{\beta \theta} |f(r \xi)|^{\theta} dr d\xi \right)^{1/\theta} \\
& =N \left(\int_{|x| \leq a}\Big( |x|^{\beta -\frac{d-1}{\theta}}|f(x)|\Big)^\theta dx \right)^{1/\theta},
\end{split}
\]
for $N = N(d, \beta, \theta)>0$. The proof is completed.
\end{proof}
We now state and prove the following lemma which provides the first step in the proof of Theorem \ref{thrm1}.
\begin{lemma} \label{L-p-trace}  Let  $\vec{p} \in [1, \infty)^d, q \in [1, \infty)$, $\alpha \in (-q, q-1)$. There exists $N = N(q,\alpha)>0$ such that
\[
\|u(\cdot, 0)\|_{L_{\vec{p}}(\bR^d)} \leq N \|u\|_{W^{1}_{\vec{p}, q}(\bR^d  \times (0,1),\mu)}
\]
for every $u \in W^{1}_{\vec{p}, q}(\bR^d  \times (0,1),\mu) \cap C(\bR^d \times [0, 1))$. Moreover, if $\alpha  \leq -1$ then $u(\cdot, 0) =0$ in sense of trace for every $u \in W^{1}_{\vec{p}, q}(\bR^d  \times (0,1),\mu)$.
\end{lemma}
\begin{proof} We prove the first assertion in the lemma. Using the density, we may assume that $u \in C^1(\bR^d  \times [0,1]) $. For each $(x, y) \in \bR^d  \times (0,1)$, by the fundamental theorem of calculus, we have
\[
u(x, y) - u(x, 0) = \int_{0}^{y} D_{d+1} u(x', s) ds
\]
Then, by Minkowski's inequality, we obtain
\begin{align} \notag
\|u(\cdot, 0)\|_{L_{\vec{p}}(\bR^d)} & \leq \|u(\cdot,y)\|_{L_{\vec{p}}(\bR^d)} +\Big \|\int_{0}^{y} |D_{d+1} u(x', s )| ds \Big \|_{L_{\vec{p}}(\bR^d)}\\ \notag
& \leq \|u(\cdot,y)\|_{L_{\vec{p}}(\bR^d)} + \int_{0}^y  \|D_{d+1} u(\cdot, s)\|_{L_{\vec{p}}(\bR^d)} ds \\ \notag
& \leq \|u(\cdot,y)\|_{L_{\vec{p}}(\bR^d)} + N y^{1-\frac{\alpha+1}{q}}\left(\int_{0}^y \|D_{d+1} u(\cdot, s)\|_{L_{\vec{p}}(\bR^d)}^q s^{\alpha} ds\right)^{1/q} \\ \label{1229-eqn}
& = \|u(\cdot,y)\|_{L_{\vec{p}}(\bR^d)} + N y^{1-\frac{\alpha+1}{q}}\|D_{d+1} u\|_{L_{\vec{p}, q}(\bR^d \times (0, y), \mu)},
\end{align}
where $N = N(q, \alpha) >0$, and we also used H\"{o}lder's inequality and the fact that $\alpha < q-1$ in the third step. Now,  as $\alpha + q >0$, $\displaystyle{\int_{0}^1 y^{\frac{\alpha}{q}} dy = \frac{q}{\alpha+q}}$. We have
\[
\begin{split}
\|u(\cdot, 0)\|_{L_{\vec{p}}(\bR^d)} & \leq N \int_0^1 \|u(\cdot,y)\|_{L_{\vec{p}}(\bR^d)} y^{\frac{\alpha}{q}} dy + N  \|D_{d+1} u\|_{L_{\vec{p}, q}(Q, \mu)}\int_0^1 y^{1-\frac{1}{q}} dy \\
& \leq N \|u\|_{L_{\vec{p}, q}(\bR^d \times (0, 1), \mu)} + N \|D_{d+1} u\|_{L_{\vec{p}, q}(\bR^d \times (0, 1), \mu)} \\
& \leq N \|u\|_{W^1_{\vec{p},q}(\bR^d \times (0, 1), \mu)}.
\end{split}
\]
where $N =N(q, \alpha)>0$.  This proves the first assertion of the lemma.

\smallskip
Next, we prove the second assertion of the lemma. Let $(x,y) \in \bR^d \times (0, 1)$. For every $s \in (0, y)$, similar to \eqref{1229-eqn}, we have
\[
\begin{split}
\|u(\cdot,y)\|_{L_{\vec{p}}(\bR^d)} & \leq \|u(\cdot,s)\|_{L_{\vec{p}}(\bR^d)} +  \left\|\int_s^y |D_{d+1} u(x, \tau)| d\tau \right\|_{L_{\vec{p}}(\bR^d)}\\
& \leq  \|u(\cdot,s)\|_{L_{\vec{p}}(\bR^d)}  + Ny^{1-\frac{1+\alpha}{q}} \|D_{d+1} u\|_{L_{\vec{p}, q}(\bR^d \times (0, y), \mu)}.
\end{split}
\]
Then, integrating this last estimate with respect to $s$ on $(0, y)$ and then using the H\"{o}lder's inequality, we have
\[
\begin{split}
y\|u(\cdot,y)\|_{L_{\vec{p}}(\bR^d)} & \leq \int_0^y \|u(\cdot,s)\|_{L_{\vec{p}}(\bR^d)} ds  + Ny^{2-\frac{1+\alpha}{q}} \|D_{d+1} u\|_{L_{\vec{p}, q}(\bR^d \times (0, y), \mu)} \\
& \leq N y^{1-\frac{1+\alpha}{q}}\left[ \|u\|_{L_{\vec{p}, q}(\bR^d \times (0, y), \mu)} + y  \|D_{d+1} u\|_{L_{\vec{p}, q}(\bR^d \times (0, y), \mu)}\right],
\end{split}
\]
where $N = N(\alpha, q)>0$. Then,
\[
\|u(\cdot,y)\|_{L_{\vec{p}}(\bR^d)} \leq N  y^{-\frac{1+\alpha}{q}}\left[ \|u\|_{L_{\vec{p}, q}(\bR^d \times (0, y), \mu)} + y  \|D_{d+1} u\|_{L_{\vec{p}, q}(\bR^d \times (0, y), \mu)}\right].
\]
From this, as $\alpha \leq -1$, and
\[
\lim_{y\rightarrow 0^+}   \left[ \|u\|_{L_{\vec{p}, q}(\bR^d \times (0, y), \mu)} + y  \|D_{d+1} u\|_{L_{\vec{p}, q}(\bR^d \times (0, y), \mu)}\right] =0,
\]
we see that
\[
  \lim_{y\rightarrow 0^+}\|u(\cdot,y)\|_{L_{\vec{p}}(\bR^d)}  =0. 
\]
The proof of the lemma is completed.
\end{proof}
\subsection{Mixed-norm Besov space equivalent norms} For $x_0 \in \bR^d$ and $\rho>0$, we denote $B_\rho(x_0)$ the ball in $\bR^d$ of radius $\rho$ centered at $x_0$. We also write $B_\rho = B_\rho(0)$. Let $\varphi \in C_0^\infty(B_1)$ be a fixed radial cut-off function satisfying
\[
0 \leq \varphi \leq 1 \quad\text{and} \quad \int_{\bR^d} \varphi(x) dx =1.
\]
For each $\delta>0$, let 
\begin{equation} \label{varphi-delta}
\varphi_\delta(x) = \delta^{-d} \varphi(x/\delta), \quad x \in \bR^d.
\end{equation}
We also denote
\begin{equation} \label{Lp-oss}
\omega(\delta, f)_{\vec{p}} = \sup_{|h| <\delta} \|\Delta_h f\|_{L_{\vec{p}}(\mathbb{R}^d)}.
\end{equation}
We now recall the following classical lemma.
\begin{lemma} \label{convol-lemma} There exists $N = N(d)>0$ such that for every $\vec{p} \in [1, \infty)^d$, we have
\[ \|\varphi_\delta *f - f\|_{L_{\vec{p}}(\mathbb{R}^d)} \leq \omega(\delta, f)_{\vec{p}}, \quad \text{and} \quad \|D_i\varphi_\delta *f \|_{L_{\vec{p}}(\mathbb{R}^d)} \leq N\omega(\delta, f)_{\vec{p}},
\]
for $\delta>0$ and $i=1, 2,\ldots, d$, where $*$ denotes the convolution operator.
\end{lemma}
\begin{proof} We provide the proof of completeness. Observe that
\[
 \varphi_\delta *f(x) - f(x) = \int_{B_1}[f(x - \delta z) - f(x)] \varphi(z) dz = \int_{B_1} \Delta_{-\delta z} f(x) \varphi(z) dz,
\]
for $x \in \bR^d$. Then, by Minkowski's inequality, we obtain
\[
\|\varphi_\delta *f - f\|_{L_{\vec{p}}(\mathbb{R}^d)} \leq \int_{B_1} \|\Delta_{-\delta z} f\|_{L_{\vec{p}}(\bR^d)} \varphi(z) dz \leq \omega(\delta, f)_{\vec{p}}. 
\]
which proves the first assertion of the lemma. As $\varphi$ is radial, 
\[
\int_{B_1} D_i \varphi (z) dz = 0.
\]
Then, we also have
\[
  D_i\varphi_\delta *f (x) = \int_{B_1} \Delta_{-\delta z} f(x) D_i\varphi(z) dz.
\]
and the second assertion of the lemma follows in the same way.
\end{proof}
We now state and prove the following result on the equivalent norms of the mixed-norm Besov spaces which will be used later in the proof of Theorem \ref{thrm1}.
\begin{lemma} \label{equiv-norm}  Let $\ell \in (0, 1)$, $\vec{p}, \theta \in [1, \infty)^d$, and $a \in [1, \infty]$. Define 
\begin{equation} \label{B-1-norm}
\|f\|_{B^{\ell}_{\vec{p}, \theta}(\bR^d)}^{(1)} = \|f\|_{L_{\vec{p}}(\bR^d)} + \left( \int_0^{a} \left( \frac{\omega(t, f)_{\vec{p}}}{t^\ell} \right)^\theta \frac{dt}{t} \right)^{1/\theta}
\end{equation}
and
\begin{equation} \label{B-2-norm}
\|f\|_{B^{\ell}_{\vec{p}, \theta}(\bR^d)}^{(2)} = \|f\|_{L_{\vec{p}}(\bR^d)} + \left( \sum_{k=1}^\infty\left(2^{k\ell}  \omega(2^{-k}, f)_{\vec{p}}  \right)^\theta \right)^{1/\theta}.
\end{equation}
Then, the norms $\|\cdot\|_{B^{\ell}_{\vec{p}, \theta}(\bR^d)}^{(1)}$, $\|\cdot\|_{B^{\ell}_{\vec{p}, \theta}(\bR^d)}^{(2)}$, and $\|\cdot \|_{B^{\ell}_{\vec{p}, \theta}(\bR^d)}$ are equivalent.
\end{lemma}
\begin{proof} The result is well-known in the un-mixed norm case, but does not seem to be recorded elsewhere in the mixed-norm case. In addition, the proof is not too long, hence we provide it for completeness. Note that $\|\Delta_h f\|_{L_{\vec{p}}(\bR^d)} \leq \omega(|h|, f)_{\vec{p}}$. Therefore, it is follow from the spherical coordinate and the fact $\omega(t, f)_{\vec{p}} \leq 2\|f\|_{L_{\vec{p}}(\bR^d)}$ that
\[
\begin{split}
& \|f\|_{b^{\ell}_{\vec{p},\theta}(\bR^d)}  =\left[ \int_{\bR^d} \left( \frac{\|\Delta_h f\|_{L_{\vec{p}}(\bR^d)}}{|h|^{\ell}} \right)^\theta \frac{d h}{|h|^d} \right]^{\frac{1}{\theta}} \leq \left[ \int_{\bR^d} \left( \frac{\omega(|h|, f)_{\vec{p}}}{|h|^{\ell}} \right)^\theta \frac{d h}{|h|^d} \right]^{\frac{1}{\theta}} \\
& \leq N \left[ \int_{0}^a \left( \frac{\omega(t, f)_{\vec{p}}}{t^{\ell}} \right)^\theta \frac{d t}{t} \right]^{\frac{1}{\theta}} + N \left[ \int_{a}^\infty \left( \frac{\omega(t, f)_{\vec{p}}}{t^{\ell}} \right)^\theta \frac{d t}{t} \right]^{\frac{1}{\theta}} \\
& \leq N \left[ \int_{0}^a \left( \frac{\omega(t, f)_{\vec{p}}}{t^{\ell}} \right)^\theta \frac{d t}{t} \right]^{\frac{1}{\theta}} + N \|f\|_{L_{\vec{p}}(\bR^d)} \Big(\int_{1}^\infty t^{-\theta\ell -1} dt \Big)^{1/\theta}\\
& \leq N \left[ \int_{0}^a \left( \frac{\omega(t, f)_{\vec{p}}}{t^{\ell}} \right)^\theta \frac{d t}{t} \right]^{\frac{1}{\theta}} + N \|f\|_{L_{\vec{p}}(\bR^d)},
\end{split}
\]
where $N = N(d, \theta, \ell)>0$. From this and \eqref{B-1-norm}, there is $N = N(d, \theta, \ell)>0$ such that
\[
\|f\|_{B^{\ell}_{\vec{p}, \theta}(\bR^d)} \leq N \|f\|_{B^{\ell}_{\vec{p}, \theta}(\bR^d)}^{(1)}.
\]
This proves one direction in the the assertion about the equivalence between $\|~\cdot~\|_{B^{\ell}_{\vec{p}, \theta}(\bR^d)}$ and $\|~\cdot~\|_{B^{\ell}_{\vec{p}, \theta}(\bR^d)}^{(1)}$. 
Now, we prove the other direction. By a simple manipulation, we see that
\[
\Delta_h f(x) = \Delta_\eta f(x) + \Delta_{h-\eta} f(x+ \eta), \quad \forall \eta, h \in \bR^d.
\]
It then follows from the triangle inequality that
\[
\|\Delta_h f\|_{L_{\vec{p}}(\bR^d)} \leq \|\Delta_\eta f\|_{L_{\vec{p}}(\bR^d)} + \|\Delta_{h-\eta} f\|_{L_{\vec{p}}(\bR^d)}.
\]
Integrating this inequality with respect to $\eta$ in the ball $B_{|h|/2} (h/2)$, we obtain
\[
\begin{split}
& \|\Delta_h f\|_{L_{\vec{p}}(\bR^d)} \\
& \leq \frac{N}{|h|^d} \left( \int_{B_{|h|/2} (h/2)} \|\Delta_\eta f\|_{L_{\vec{p}}(\bR^d)} d\eta + \int_{B_{|h|/2} (h/2)}\|\Delta_{h-\eta} f\|_{L_{\vec{p}}(\bR^d)} d\eta \right) \\
& \leq \frac{N}{|h|^d} \int_{B_{|h|}} \|\Delta_\eta f\|_{L_{\vec{p}}(\bR^d)} d\eta \leq N \int_{B_{|h|}} |\eta|^{-d}\|\Delta_\eta f\|_{L_{\vec{p}}(\bR^d)} d\eta
\end{split}
\]
where $N = N(d)>0$, and we used the fact that $B_{|h|/2} (h/2) \subset B_{|h|}$ and $h-\eta \in B_{|h|}$ for all $\eta \in B_{|h|/2} (h/2)$. As a consequence, we obtain 
\[
\omega(t, f)_{\vec{p}} \leq N \int_{|\eta| \leq t} |\eta|^{-d}\|\Delta_\eta f\|_{L_{\vec{p}}(\bR^d)} d\eta, \quad \forall \ t >0
\]
and therefore
\[
\begin{split}
& \left[ \int_{0}^a \left( \frac{\omega(t, f)}{t^{\ell}} \right)^\theta \frac{d h}{t} \right]^{\frac{1}{\theta}} \\
 & \leq N \left(\int_{0}^a \Big[ t^{-\ell -\frac{1}{\theta}} \int_{|\eta| \leq t} |\eta|^{-d}\|\Delta_\eta f\|_{L_{\vec{p}}(\bR^d)} d\eta \Big]^{\theta} dt\right)^{1/\theta}.
 \end{split}
\]
Now, applying Corollary \ref{Hardy-cor} with $\beta =d -\ell -\frac{1}{\theta}$, we obtain
\[
\begin{split}
& \left[ \int_{0}^a \left( \frac{\omega(t, f)}{t^{\ell}} \right)^\theta \frac{d h}{t} \right]^{\frac{1}{\theta}} \\
& \leq N \left(\int_{B_a} \Big(|\eta|^{-\ell -\frac{d}{\theta}} \|\Delta_\eta f\|_{L_{\vec{p}}(\bR^d)} \Big)^{\theta} d\eta \right)^{1/\theta} \leq N\|f\|_{b^{\ell}_{\vec{p}, \theta}}.
\end{split}
\]
Therefore, there is $N = N(d, \ell, \theta)>0$ such that
\[
\|f\|_{B^{\ell}_{\vec{p}, \theta}(\bR^d)}^{(1)} \leq N \|f\|_{B^{\ell}_{\vec{p}, \theta}(\bR^d)}
\]
and this completes the proof that the norms $\|~\cdot~\|_{B^{\ell}_{\vec{p}, \theta}}^{(1)}$ and $\|~\cdot~\|_{B^{\ell}_{\vec{p}, \theta}}$ are equivalent.

\smallskip
Next, we consider the norm $\|\cdot \|_{B^{\ell}_{\vec{p}, \theta}}^{(2)}$ defined in \eqref{B-2-norm}. We will show that $\|\cdot \|_{B^{\ell}_{\vec{p}, \theta}}^{(2)}$ is equivalent to $\|\cdot \|_{B^{\ell}_{\vec{p}, \theta}}^{(1)}$ defined in \eqref{B-1-norm}. Observe that
\[
\begin{split}
& \left( \int_0^{a} \left( \frac{\omega(t, f)_{\vec{p}}}{t^\ell} \right)^\theta \frac{dt}{t} \right)^{1/\theta}\\
& =\left( \int_0^{1/2} \left( \frac{\omega(t, f)_{\vec{p}}}{t^\ell} \right)^\theta \frac{dt}{t} \right)^{1/\theta} + \left( \int_{1/2}^{a} \left( \frac{\omega(t, f)_{\vec{p}}}{t^\ell} \right)^\theta \frac{dt}{t} \right)^{1/\theta}.
\end{split}
\]
It follows directly from the definition of $\omega(t, f)_{\vec{p}}$ and the triangle inequality that $\omega(t, f)_{\vec{p}} \leq 2\|f\|_{L_{\vec{p}}(\bR^d)}$ and therefore,
\[
\begin{split}
\left( \int_{1/2}^{a} \left( \frac{\omega(t, f)_{\vec{p}}}{t^\ell} \right)^\theta \frac{dt}{t} \right)^{1/\theta} & \leq 2 \|f\|_{L_{\vec{p}}(\bR^d)} \left(\int_{1/2} ^\infty t^{-\ell \theta -1}dt\right)^{1/\theta} \\
& = N \|f\|_{L_{\vec{p}}(\bR^d)},
\end{split}
\]
for $N = N(\theta, \ell)>0$. On the other hand, as $\omega(t, f)_{\vec{p}}$ is non-decreasing, we see that
\[
\begin{split}
 \left( \int_0^{1/2} \left( \frac{\omega(t, f)_{\vec{p}}}{t^\ell} \right)^\theta \frac{dt}{t} \right)^{1/\theta} & =\left(\sum_{k=1}^\infty \int_{2^{-k-1}}^{2^{-k}} \left( \frac{\omega(t, f)_{\vec{p}}}{t^\ell} \right)^\theta \frac{dt}{t} \right)^{1/\theta} \\
& \leq N\left( \sum_{k=1}^\infty\left(2^{k\ell}  \omega(2^{-k}, f)_{\vec{p}}  \right)^\theta \right)^{1/\theta}.
\end{split}
\]
Then, from \eqref{B-1-norm}, it follows that 
\[
\|f\|_{B^{\ell}_{\vec{p}, \theta}(\bR^d)}^{(1)} \leq N \|f\|_{B^{\ell}_{\vec{p},\theta}(\bR^d)}^{(2)},
\]
where $N = N(d, \ell, \theta)>0$. It now remains to prove the other direction of the inequality. By Minkowski's inequality, it follows that
\[
\begin{split}
& \left( \sum_{k=1}^\infty\left(2^{k\ell}  \omega(2^{-k}, f)_{\vec{p}}  \right)^\theta \right)^{1/\theta} = N \left[ \omega(2^{-1}, f)_{\vec{p}} +  \left( \sum_{k=2}^\infty\left(2^{k\ell}  \omega(2^{-k}, f)_{\vec{p}}  \right)^\theta \right)^{1/\theta} \right] \\
&\leq N \left[ \|f\|_{L_{\vec{p}}(\bR^d)}  +   \left( \int_0^{1/2} \left( \frac{\omega(t, f)_{\vec{p}}}{t^\ell} \right)^\theta \frac{dt}{t} \right)^{1/\theta} \right].
\end{split}
\]
Therefore,
\[
\|f\|_{B^{\ell}_{\vec{p}, \theta}(\bR^d)}^{(2)} \leq N \|f\|_{B^{\ell}_{\vec{p}, \theta}(\bR^d)}^{(1)}
\]
where $N = N(d, \theta, \ell)>0$. The proof of the lemma is completed.
\end{proof}
\section{Proof of Theorem \ref{thrm1}}  \label{Sec-2}

This section is to prove Theorem \ref{thrm1}. The proof requires several lemmas and estimates.  For $k = 1, 2, \ldots, d$, and for $h \in \mathbb{R}$, let us denote the difference in $k$-variable direction of a function $f$ as
\[
\Delta_{k, h} f(x) = f(x + h\mathbf{e}_k) - f(x), \quad x \in \bR^d,
\]
where $\mathbf{e}_k$ is the Euclidean $k^{\textup{th}}$-unit coordinate vector. We start with the one dimensional space as its proof contains important ideas and techniques.
\begin{lemma} Let $p, q \in [1, \infty)$ and $\alpha \in (-1, q-1)$.  Then, there is $N = N(p, q, \alpha)>0$ such that
\[
\|f(\cdot, 0)\|_{b^{\ell}_{p, q}(\bR)} \leq N\left( \int_{0}^\infty \left(\int_{\bR} |D f(x, y)|^p dx\right)^{\frac{q}{p}}  y^{\alpha} dy  \right)^{1/q}
\]
for all $f \in C(\bR \times [0, \infty))$ such that $Df \in L_{p, q}(\bR^d_+, \mu)$ and $\ell = 1-\frac{1+\alpha}{q} \in (0, 1)$.
\end{lemma}
\begin{proof} Let $g(x) = f(x, 0)$ for $x \in \bR$ and note that
\[
\begin{split}
\Delta_h g(x) & = [f(x+h, 0) - f(x+h, |h|)]  + [f(x+h, |h|) - f(x, |h|)] \\
& \qquad + [f(x, |h|) - f(x,0)] \\
& = -\Delta_{2,|h|}f(x+h, 0) + \Delta_{1,h} f(x, |h|) + \Delta_{2,|h|} f(x, 0), \quad x \in \bR.
\end{split}
\]
Then, by triangle inequality, it follows that
\[
\begin{split}
\| \Delta_h g\|_{L_p(\bR)}& \leq  \left(\int_{\bR}|\Delta_{2,|h|}f(x+h, 0)|^p dx  \right)^{\frac{1}{p}} +  \left(\int_{\bR}| \Delta_{1,h} f(x, |h|)|^p dx  \right)^{\frac{1}{p}} \\
& \qquad +  \left(\int_{\bR}|  \Delta_{2,|h|} f(x, 0)|^p dx  \right)^{\frac{1}{p}} \\
& =2  \left(\int_{\bR}|\Delta_{2,|h|}f(x, 0)|^p dx  \right)^{\frac{1}{p}} +  \left(\int_{\bR}| \Delta_{1,h} f(x, |h|)|^p dx  \right)^{\frac{1}{p}}.
\end{split}
\]
Then, for $\ell = 1 -\frac{1+\alpha}{q}$, we have
\[
\begin{split}
& \left( \int_{\bR} \Big(\frac{ \| \Delta_h g\|_{L_p(\bR)}}{|h|^\ell} \Big)^q \frac{dh}{|h|} \right)^{\frac{1}{q}} \\
& \leq N \left( \int_0^\infty h^{-1-q\ell }  \left(\int_{\bR}|\Delta_{2,h}f(x, 0)|^p dx  \right)^{\frac{q}{p}} dh \right)^{1/q} \\
& \qquad + N \left( \int_0^\infty h^{-1-q\ell} \left(\int_{\bR}| \Delta_{1,h} f(x, h)|^p dx  \right)^{\frac{q}{p}} dh \right)^{1/q} \\
& = N [I_1 + I_2],
\end{split}
\]
where $N = N(q) >0$. We now control $I_1$. From the fundamental theorem of calculus, we have
\[
\Delta_{2,h}f(x,0) = f(x, h) - f(x,0) = \int_0^h D_2 f(x, \xi) d\xi
\]
and therefore it follows from Minkowski's inequality that
\[
\begin{split}
\left(\int_{\bR}|\Delta_{2,h}f(x, 0)|^p dx  \right)^{\frac{1}{p}} & \leq \left(\int_{\bR}\Big (\int_0^h |D_2 f(x, \xi)| d\xi \Big)^p dx  \right)^{\frac{1}{p}} \\
& \leq \int_0^h \left(\int_{\bR} |D_2f(x, \xi)|^p dx \right)^{1/p} d\xi = \int_0^h w(\xi) d\xi.
\end{split}
\]
where $w(\xi) = \displaystyle{\left(\int_{\bR} |D_2f(x, \xi)|^p dx \right)^{1/p}}$. Then, we see that
\[
\begin{split}
I_1 \leq \left(\int_0^\infty h^{-1-\ell q} \Big(\int_0^h w(\xi) d\xi \Big)^q \right)^{1/q} = \left(\int_0^\infty h^{q(1-\ell - 1/q)} \Big(\frac{1}{h}\int_0^h w(\xi) d\xi \Big)^q \right)^{1/q}
\end{split}
\]
Now as $1-\ell -1/q = \frac{\alpha}{q} < \frac{1}{q'} = \frac{q-1}{q}$, it follows from the weighted Hardy's inequality stated in Lemma \ref{Hardy-ineq}  that
\[
\begin{split}
I_1 \leq N \left(\int_0^\infty |w(y)|^q y^\alpha dy \right)^{1/q} = N\left(\int_0^\infty \left(\int_{\bR} |D_2f(x, y)|^p dx \right)^{q/p} y^\alpha dy\right)^{1/q}.
\end{split}
\]
Next, we control $I_2$, which is defined as
\[
I_2 = \left( \int_0^\infty h^{-1-q\ell} \left(\int_{\bR}| \Delta_{1,h} f(x, h)|^p dx  \right)^{\frac{q}{p}} dh \right)^{1/q}.
\]
By the fundamental theorem of calculus, we have
\[
\Delta_{1, h} f(x, h) = f(x+h, h) - f(x, h) = \int_{0}^{h} D_1f(x+ s, h) ds.
\]
Then,
\[
\begin{split}
\left(\int_{\bR}| \Delta_{1,h} f(x, h)|^p dx  \right)^{\frac{1}{p}} & \leq \left(\int_{\bR}\Big(  \int_{0}^{h} |D_1f(x+s, h)| ds \Big)^p dx  \right)^{\frac{1}{p}} \\
& \leq \int_{0}^h\left(  \int_{\bR} | D_1f(x+s, h)|^p dx \right)^{\frac{1}{p}} ds \\
& = h \|D_1(\cdot, h)\|_{L_p(\bR)},\end{split}
\]
where we have used Minkowski's inequality in second step of the previous estimate. As a result, we see that
\[
\begin{split}
I_2 & \leq  \left( \int_0^\infty  h^{q(1-\ell-1/q)} \left(  \int_{\bR} |D_1f(x, h)|^p dx \right)^{\frac{q}{p}} dh \right)^{\frac{1}{q}} \\
& = \left( \int_0^\infty  \left(  \int_{\bR} |D_1f(x, h)|^p dx \right)^{\frac{q}{p}} h^{\alpha}  dh \right)^{\frac{1}{q}}.
\end{split}
\]
Collecting the estimates of $I_1$ and $I_2$, we see that
\[
\begin{split}
& \|g\|_{b^{\ell}_{p,q}(\bR)} = \left( \int_{\bR} \Big(\frac{ \| \Delta_h g\|_{L_p(\bR)}}{|h|^\ell} \Big)^q \frac{dh}{|h|} \right)^{\frac{1}{q}}    \leq N \|D f\|_{L_{p, q}(\bR^2_+, \mu)},
\end{split}
\]
and the proof is then completed.
\end{proof}

Next, we prove the trace inequality in the multi-dimensional space. 
\begin{lemma} \label{b-p-trace} Let $\vec{p} \in [1, \infty)^d, q \in [1, \infty)$ and $\alpha \in (-1, q-1)$.  Then, there is $N = N(d, \vec{p}, q, \alpha)$ such that
\begin{equation} \label{trace-Rn}
\|f(\cdot, 0)\|_{b^{\ell}_{\vec{p}, q}(\bR^d)} \leq N\| Df\|_{L_{\vec{p}, q}(\bR^{d+1}_+, \mu)},
\end{equation}
for all $f \in C(\bR^d \times [0, \infty))$ such that $Df \in L_{\vec{p}, q}(\bR^{d+1}_+, \mu)$ and for $\ell = 1-\frac{1+\alpha}{q} \in (0,1)$.
\end{lemma}
\begin{proof} Let $g(x) = f(x, 0),  x \in \bR^d$. For $x, h \in \bR^d$ and $y, s \in [0, \infty)$, we write
\[
\begin{split}
&\Delta_{d+1, s} f(x, y) = f(x, y+s) - f(x,y), \quad \text{and} \\
& \Delta_h f(x,y) = f(x+h, y) - f(x, y).
\end{split}
\]
We have
\[
\begin{split}
\Delta_h g(x) & = f(x+h, 0) - f(x,0) \\
& =-[f(x,0) - f(x, |h|)] + [f(x+h, |h|) - f(x, |h|)] \\
& \qquad - [f(x+h, |h|) - f(x+h, 0)]\\
& = -\Delta_{d+1, |h|} f(x,0) + \Delta_{h} f(x, |h|) - \Delta_{d+1, |h|} f(x+h,0).
\end{split} 
\]
Then, by the triangle inequality, we have
\[
\begin{split}
\|\Delta_h g\|_{L_{\vec{p}}(\bR^d)} &  \leq \| \Delta_{d+1, |h|} f(\cdot,0)\|_{L_{\vec{p}}(\bR^d)} +\| \Delta_{h} f(\cdot, |h|) \|_{L_{\vec{p}}(\bR^d)} \\
& \qquad + \| \Delta_{d+1, |h|} f(\cdot+h,0) \|_{L_{\vec{p}}(\bR^d)} \\
& = 2 \| \Delta_{d+1, |h|} f(\cdot,0)\|_{L_{\vec{p}}(\bR^d)} +\| \Delta_{h} f(\cdot, |h|) \|_{L_{\vec{p}}(\bR^d)}.
\end{split}
\]
Hence, it follows that
\[
\begin{split}
\|g\|_{b^\ell_{{\vec{p}}, q}(\bR^d)} & =  \left[ \int_{\bR^d} \Big ( \frac{\|\Delta_h g\|_{L_{\vec{p}}(\bR^d)}}{|h|^{\ell}}\Big)^q \frac{dh}{|h|^d} \right]^{\frac{1}{q}} \\
& \leq N \left[\int_{\bR^d} |h|^{-q\ell -d}\| \Delta_{d+1, |h|} f(\cdot,0)\|_{L_{\vec{p}}(\bR^d)}^q dh \right]^{\frac{1}{q}} \\ & \qquad + N \left[\int_{\bR^d} |h|^{-q\ell -d}\| \Delta_{h} f(\cdot, |h|) \|_{L_{\vec{p}}(\bR^d)}^q dh \right]^{\frac{1}{q}}
= N[I_1 + I_2].
\end{split}
\]
By the fundamental theorem of calculus, we have
\[
\Delta_{d+1, |h|} f(x,0) = \int_0^{|h|} D_{d+1} f(x, s) ds.
\]
Then, it follows from Minkowski's inequality that
\[
\begin{split}
\| \Delta_{d+1, |h|} f(\cdot+h,0) \|_{L_{\vec{p}}(\bR^d)} & \leq \left \| \int_0^{|h|} D_{d+1} f(\cdot, s) ds\right \|_{L_{\vec{p}} (\bR^d)}\\
& \leq \int_0^{|h|}   \|D_{d+1} f(\cdot, s)\|_{L_{\vec{p}}(\bR^d)} ds \\
&= \int_0^{|h|} w(s) ds,
\end{split}
\]
where $w(s) =  \|D_{d+1} f(\cdot, s)\|_{L_{\vec{p}}(\bR^d)}$.  Therefore, 
\[
\begin{split}
I_1 & \leq    \left( \int_{\bR^d} |h|^{-q\ell -d} \Big( \int_0^{|h|}w(s) ds\Big )^q  dh \right)^{\frac{1}{q}}    \\
 & = N   \left( \int_{0}^\infty r^{-q\ell -d} \Big( \int_0^{r}w(s) ds\Big )^q r^{d-1} dr \right)^{\frac{1}{q}}   \\
& = N   \left( \int_{0}^\infty r^{q(1-\ell -1/q)} \Big( \frac{1}{r}\int_0^{r}w(s) ds\Big )^q  dr \right)^{\frac{1}{q}},
\end{split}
\]
where we have used polar coordinate in the second step of the previous estimate. Then, as $1-\ell- \frac{1}{q} = \frac{\alpha}{q} < 1- \frac{1}{q}$,  it follows from  Lemma \ref{Hardy-ineq} that
\[
\begin{split}
I_1 & \leq N   \left( \int_{0}^\infty |w(s)|^q s^\alpha ds \right)^{\frac{1}{q}} =N \|D_{d+1} f\|_{L_{\vec{p}, q}(\bR^{d+1}_+, \mu)}.
\end{split}
\]
It now remains to control $I_2$.  Again, from the fundamental theorem of calculus, it follows that
\[
\Delta_h f(x, |h|) = f(x+h, |h|) - f(x, |h|) = \int_0^1 D_x f(x+ sh, |h|) \cdot hds
\]
where we denote
$$D_x f(x, y) = (D_{x_1} f(x,y), \ldots, D_{x_d} f(x,y)), \quad  (x, y) \in \bR^{d} \times \bR_+.$$ 
Then, by the Minkowski inequality, we infer that
\[
\begin{split}
\|\Delta_h f(\cdot, |h|)\|_{L_{\vec{p}}(\bR^d)} & \leq |h|  \left\|  \int_0^1 |D_xf(\cdot+ sh, |h|)| ds \right\|_{L_{\vec{p}}(\bR^d)} \\
& \leq |h|  \int_0^1 \left\| D_xf(\cdot + sh, |h|) \right\|_{L_{\vec{p}}(\bR^d)} ds \\
& =|h| \left\| D_xf(\cdot, |h|) \right\|_{L_{\vec{p}}(\bR^d)}.
\end{split}
\]
Therefore, it follows that
\[
\begin{split}
I_2 & \leq \left[\int_{\bR^d} |h|^{q(1-\ell) -d} \left\| D_xf(\cdot, |h|) \right\|_{L_{\vec{p}}(\bR^d)}^q dh \right]^{\frac{1}{q}} \\
& = N \left[\int_{0}^{\infty} r^{\alpha} \left\| D_xf(\cdot, r) \right\|_{L_{\vec{p}}(\bR^d)}^q dr \right]^{\frac{1}{q}}= N  \|D_x f\|_{L_{\vec{p}, q}(\bR^{d+1}_+, \mu)}.
\end{split}
\]
Finally, by collecting the estimates of the two terms $I_1$ and $I_2$, we obtain \eqref{trace-Rn}. The proof is then completed.
\end{proof}
Our next step is to show that for a given $g \in B^{\ell}_{\vec{p},q}(\bR^d)$ with $\ell = 1 - \frac{1+\alpha}{q}$, $\alpha \in (-1, q-1)$, $\vec{p} \in [1, \infty)^d$, and $q \in [1, \infty)$, we can extend $g$ to a function $u \in W^{1}_{\vec{p}, q} (\bR^{d+1}_+, \mu)$. For this purpose, for each $k \in \mathbb{Z}$, let $S_k = (2^{-k-1}, 2^{-k})$ and consider the family of partition of unity $\{\psi_k\}_{k \in \mathbb{Z}}$ satisfying $\psi_k \in C^\infty_0(\bR)$, $0 \leq \psi_k \leq 1$, 
\begin{equation} \label{psi-support}
 S_k \subset \text{spt}(\psi_k) \subset (7/2^{k+4}, 9/2^{k+3}) \subset S_{k-1} \cup S_k \cup S_{k+1}
\end{equation}
for all $k \in \mathbb{N}$. Moreover, there is $N_0>0$ such that
\begin{equation} \label{psi-derivetive}   
\|D^\gamma\psi_k\|_{L_\infty(\bR)} \leq N_0 2^{k \gamma}, \quad \forall \ \gamma \in \mathbb{N}, \quad \forall \ k \in \mathbb{Z}
\end{equation}
and
\begin{equation} \label{sum-partition}
\sum_{k\in\mathbb{Z}} \psi_k(y) =1, \quad \forall \ y \in (0, \infty).
\end{equation}
We note that the existence of such a family of partition of unity $\{\psi_k\}_{k \in \mathbb{Z}}$ is well-known, see \cite[Lemma 5, p. 43]{Burenkov} for instance.
As in \eqref{varphi-delta}, let us recall that $\varphi \in C_0^\infty(\bR^d)$ is a standard cut-off function which is radial and it satisfies
\[
0 \leq \varphi \leq 1, \quad \int_{\bR^d} \varphi(x) dx =1, \quad \text{and} \quad \textup{supp}(\varphi) \subset B_1.
\]
We denote $\varphi_k (x) = 2^{kd}\varphi(2^{k}x)$ 
and
\[
\mathcal{A}_{k,\varphi} (g)(x) = \int_{\bR^d} g(z) \varphi_k(x-z) dz, \quad k \in \mathbb{Z}.
\]
For a given $g \in B^{\ell}_{\vec{p},q}(\bR^d)$, we define the following Burenkov's extension map
\begin{equation} \label{E.def}
\oE(g)(x,y) = \sum_{k=1}^\infty \psi_k(y) \mathcal{A}_{k, \varphi}(g)(x), \quad (x,y) \in \bR^{d+1}_+.
\end{equation}

\begin{lemma}[Boundedness of the extension map] \label{E-lemma} Let $\vec{p} \in [ 1, \infty)^d, q \in [ 1, \infty)$, $\alpha \in (-1, q-1)$, $\ell =1-\frac{1+\alpha}{q}$, and let $\oE$ be defined as in \eqref{E.def}. Then $\oE: B^{\ell}_{\vec{p},q}(\bR^d) \rightarrow W^1_{\vec{p}, q}(\bR^{d+1}_+, \mu)$ and there exists $N = N(\vec{p}, q, \alpha, d)>0$ such that
\begin{equation} \label{bounded-E}
\|\oE(g)\|_{W^1_{\vec{p}, q}(\bR^{d+1}_+, \mu)} \leq N \|g\|_{B^{\ell}_{\vec{p},q}(\bR^d)}, \quad \forall~g \in B^{\ell}_{\vec{p},q}(\bR^d).
\end{equation}
Moreover,
\begin{equation} \label{lim-eqn}
\lim_{y\rightarrow 0^+} y^{-\ell}\|\oE(g)(\cdot, y) - g\|_{L_{\vec{p}}(\bR^d)} =0 \quad \text{and} \quad \textup{spt}(\oE(g)(x, \cdot)) \subset [0, 1) \quad \forall \ x \in \bR^d.
\end{equation}
\begin{proof} From \eqref{psi-support}, we note that $\text{spt}(\psi_k) \subset (0, 1)$ when $k \in \mathbb{N}$. Then,
\begin{equation} \label{sup-E}
\text{spt}(\oE(g)(x, \cdot)) \subset [0, 1], \quad \forall \ x \in \bR^d.
\end{equation}
On the other hand, from mixed-norm Young's inequality (see \cite[Theorem 2.1]{P} for instance)  that
\[
\|\mathcal{A}_{k, \varphi}(g)\|_{L_{\vec{p}}(\bR^d)} \leq \|\varphi_k\|_{L_1(\bR^d)} \|g\|_{L_{\vec{p}}(\bR^d)} = \|g\|_{L_{\vec{p}}(\bR^d)} .
\]
Therefore, for each $y >0$, by Minkowski's inequality and \eqref{sum-partition} 
\[
\begin{split}
\|\oE(g)(\cdot, y)\|_{L_{\vec{p}}(\bR^d)} & \leq \sum_{k=1}^\infty \psi_k(y) \|\mathcal{A}_{k, \varphi}(g)\|_{L_{\vec{p}}(\mathbb{R}^d)} \\
& \leq \|g\|_{L_{\vec{p}}(\bR^d)} \sum_{k=1}^\infty \psi_k(y) \leq  \|g\|_{L_{\vec{p}}(\bR^d)}.
\end{split}
\]
From this and due to \eqref{sup-E} and  $\alpha >-1$
\begin{align} \notag
\|\oE(g)\|_{L_{\vec{p}, q}(\bR^{d+1}_+, \mu)} & \leq\|g\|_{L_{\vec{p}}(\bR^d)}\Big( \int_0^1 y^\alpha dy\Big)^{1/q} = N \|g\|_{L_{\vec{p}}(\bR^d)}\\ \label{Lp-E}
& \leq   N  \|g\|_{B^{\ell}_{\vec{p},q}(\bR^d)},
\end{align}
where $N = N(\alpha, q)>0$. Next, we control the $L_{\vec{p},q}(\bR^{d+1}, \mu)$-norms of the derivatives of $\oE(g)$. We note that from the definition of the sequence $\{S_k\}_k$, the definition of $\mu$, and \eqref{psi-support}, it follows that
\[
\oE(g)(x, y) = \psi_1(y) \mathcal{A}_{1, \varphi}g(x), \quad x \in \mathbb{R}^d, \quad y > 7/16
\]
and 
\[
\|\oE(g)\|_{W^{1}_{\vec{p}, q}(\bR^d \times (7/16, \infty), \mu)} = N(\alpha) \|\oE(g)\|_{W^{1}_{\vec{p}, q}(\bR^d \times (7/16, 1))}.
\]
We note that for all $k = 0, 1, 2, \ldots, $
\[
\|D_x^{k}\oE(g)\|_{L_{\vec{p}, q}(\bR^d \times (7/16, 1)} = \|\psi_1\|_{L_q(7/16, 1)} \|D_x^k\mathcal{A}_{1, \varphi}(g)_{L_{\vec{p}}(\bR^d)} \leq N \|g\|_{L_{\vec{p}}(\bR^d)}.
\]
On the other hand,  
\[
\|D_y \oE(g)\|_{L_{\vec{p}, q}(\bR^d \times (7/16, \infty)} = \|D_y \psi_1\|_{L_q(7/16, 1)} \|\mathcal{A}_{1, \varphi}(g)\|_{L_{\vec{p}}(\bR^d)} \leq N \|g\|_{L_{\vec{p}}(\bR^d)}.
\]
Therefore, we conclude that
\begin{equation} \label{large-y}
\|D\oE(g)\|_{L_{\vec{p}, q}(\bR^d \times (7/16, \infty), \mu)} \leq N  \|g\|_{L_{\vec{p}}(\bR^d)}.
\end{equation}

\smallskip
Next, we consider the case when $0 < y \leq 7/16$.  We recall the definition of $\omega(\delta, g)_{\vec{p}}$ in \eqref{Lp-oss}. Applying Lemma \ref{convol-lemma}, we infer that
\[
\begin{split}
\|D_x\oE(g)(\cdot, y)\|_{L_{\vec{p}}(\bR^d)} & \leq \sum_{k=1}^\infty \psi_k(y) 2^{k} \|\mathcal{A}_{k, D\varphi}(g)\|_{L_{\vec{p}}(\mathbb{R}^d)} \\
& \leq N  \sum_{k=1}^\infty \psi_k(y)2^k\omega(2^{-k}, g)_{\vec{p}}.
\end{split}
\]
Then, it follows from Lemma \ref{equiv-norm} that
\begin{align} \notag
& \|D_x\oE(g)\|_{ L_{\vec{p}, q}(\bR^d \times (0, 7/16), \mu)} \\ \notag
& \leq N\left(\sum_{k=1}^\infty \|\psi_k\|_{L_q((0,1), \mu)}^q2^{kq}\omega(2^{-k}, g)_{\vec{p}}^q \right)^{1/q} \\ \notag 
&  \leq  N  \left(\sum_{k=1}^\infty 2^{q(1-\frac{1+\alpha}{q})k}\omega(2^{-k}, g)_{\vec{p}}^q \right)^{1/q} \leq N \|g\|_{B^{\ell}_{\vec{p},q}}^{(2)}\\ \label{Dx-2}
& \leq N \|g\|_{B^{\ell}_{\vec{p},q}},
\end{align}
where  $N = N(d, \vec{p}, q, \alpha) >0$. Now, we control $\|D_y\oE(g)\|_{L_{\vec{p}, q}(
\bR^d \times (0, 7/16),\mu)}$. For $y \in (0, 7/16)$, we see that $\psi_k(y) =0$ for all $k \in \mathbb{Z}$ and $k \leq 0$. Then, it follows from \eqref{sum-partition} that
\[
\sum_{k=1}^\infty \psi_k(y) =1\quad \text{and then therefore} \quad \sum_{k=1}^\infty \psi_k'(y) =0.
\]
Due to this 
\[
D_y \oE (g)(x,y) = \sum_{k=1}^\infty \psi_k'(y) \mathcal{A}_{k, \varphi}(g)(x) =  \sum_{k=1}^\infty \psi_k'(y) \big[\mathcal{A}_{k, \varphi}(g)(x) - g(x)\big],
\]
for each $(x,y) \in \bR^d \times (0, 7/16)$. From this, \eqref{psi-derivetive}, Lemma \ref{convol-lemma}, and the first assertion in Lemma \ref{equiv-norm},  it follows that
\begin{align} \notag
& \|D_y \oE (g)(\cdot,y)\|_{L_{\vec{p}, q}(\bR^d\times (0, 7/16), \mu)} \\\notag
& \leq N(q) \left( \sum_{k=1}^\infty \|\psi'_k(y)\|_{L_q((0, 7/16), \mu)}^q \|\mathcal{A}_{k, \varphi}(g)- g\|_{L_{\vec{p}}(\bR^d)}^q \right)^{1/q} \\ \notag
& \leq N \left( \sum_{k=1}^\infty  2^{k(1-\frac{1+\alpha}{q})q}  \omega(2^{-k}, g)_{\vec{p}}^q \right)^{1/q} \leq N \|g\|_{B_{\vec{p},q}^{\ell}}^{(2)}\\ \label{Dy-2}
& \leq N \|g\|_{B_{\vec{p},q}^{\ell}}.
\end{align}
Now, by collecting \eqref{large-y}, \eqref{Dx-2}, and \eqref{Dy-2}, we obtain
\begin{equation} \label{DE-est}
\|D\oE(g)\|_{L_{\vec{p}, q}(\bR^{d+1}_+, \mu)} \leq N  \|g\|_{B_{\vec{p},q}^{\ell}}
\end{equation} 
for $N = N(\vec{p}, q, \alpha, d) >0$. Hence, \eqref{bounded-E} follows from the estimates \eqref{Lp-E} and \eqref{DE-est}.

\smallskip
Now, observe that  second assertion in \eqref{lim-eqn} is verified in \eqref{sup-E},  it remains to prove the first assertion in \eqref{lim-eqn}. For each $y \in (0, 1/4)$, let $s = s(y) \in \mathbb{N}$ such that $2^{-(s+1)}< y < 2^{-s}$. Then, we see that
\[
\begin{split}
\|\oE(g)(\cdot, y) - g\|_{L_{\vec{p}}(\bR^d)} & = \Big\|\sum_{k=s-1}^{s+1} \psi_k(y) \mathcal{A}_{k, \varphi}(g) - g\Big\|_{L_{\vec{p}}(\bR^d)} \\
& \leq \sum_{k=s-1}^{s+1} \psi_k(y) \|  \mathcal{A}_{k, \varphi}(g) - g\|_{L_{\vec{p}}(\bR^d)} \\
& \leq N \sum_{k=s-1}^{s+1} \psi_k(y)  \omega(2^{-k}, g)_{\vec{p}}  \\
& \leq N y^{\ell} \sum_{k=s-1}^{s+1}  2^{k\ell}  \omega(2^{-k}, g)_{\vec{p}}.
\end{split}
\]
We observe  from  Lemma \ref{equiv-norm} that
\[
\lim_{k\rightarrow \infty}2^{k\ell}  \omega(2^{-k}, g)_{\vec{p}} =0.
\]
Moreover, as $y \rightarrow 0^+$ we have $s \rightarrow \infty$. Therefore,
\[
\begin{split}
& \lim_{y\rightarrow 0^+}y^{-\ell}\|\oE(g)(\cdot, y) - g\|_{L_{\vec{p}}(\bR^d)}  =0,
\end{split}
\]
and the proof of the lemma is completed.
\end{proof}
\end{lemma}
\begin{proof}[Proof of Theorem \ref{thrm1}] From Lemma \ref{L-p-trace} and Lemma \ref{b-p-trace}, it follows that the trace map $\mathsf{T} :W^{1}_{\vec{p}, q}(\bR^d_+, \mu)\rightarrow B^{\ell}_{\vec{p}, q}(\bR^d)$ is well-defined and \eqref{trace-est} holds. Therefore, (i) of Theorem \ref{thrm1} is proved. On the other hand, (ii) of Theorem \ref{thrm1} follows from Lemma \ref{E-lemma}.
\end{proof}
\section{Proofs of Corollary \ref{thrm3} and Corollary \ref{thrm2}} \label{Sec-3}
This section provides the proofs of Corollary \ref{thrm3} and Corollary \ref{thrm2}, which are applications of Theorem \ref{thrm1}. We first start with the proof of Corollary \ref{thrm3}.
\begin{proof}[Proof of Corollary \ref{thrm3}]  Let $v = \oE(g)$, where $\oE$ is defined in Theorem \ref{thrm1}. We have $v \in W^1_{\vec{p}, q}(\bR^{d+1}_+, \mu)$ and
\[
\|v\|_{W^1_{\vec{p}, q}(\bR^{d+1}_+, \mu)} \leq N \|g\|_{B_{\vec{p}, q}^{\ell}(\bR^d)}
\]
where $N = N(d, \vec{p}, q, \alpha)>0$. Let $w = u-v$, and we see that $u \in W^1_{\vec{p}, q}(\bR^{d+1}_+, \mu)$ is a weak solution of \eqref{expl-PDE} if and only if $w \in W^1_{\vec{p}, q}(\bR^{d+1}_+, \mu)$ is a weak solution of
\begin{equation} \label{PDE-trn}
\left\{
\begin{array}{cccl}
\sL w(x,y)  & = & D_i (y^\beta G_i) + \sqrt{\lambda} y^\beta \tilde{f} & \quad (x, y) \in  \bR^{d} \times \bR_+\\
w (x, 0) & =& 0 & \quad x \in  \bR^{d},
\end{array} \right. 
\end{equation}
where
\[
\tilde{f} = f - \sqrt{\lambda} v \quad \text{and} \quad G_i = F_i + a_{ij} D_j v, \quad i = 1, 2,\ldots, d+1.
\]
Now, we apply \cite[Remark 2.7]{DP21} and  \cite[Theorem 2.8]{DP21} to  \eqref{PDE-trn} to find $\delta>0$ and $\lambda_0>0$ such that the existence, uniqueness, and estimates of solutions $w \in W^1_{\vec{p}, q}(\bR^{d+1}_+, \mu)$ to the equation for \eqref{PDE-trn} follow when $\lambda \geq \lambda_0 R_0^{-2}$ and \eqref{VMO-cond} holds.  In particular, we have
\[
\|Dw\|_{L_{\vec{p}, q}(\bR^{d+1}_+, \mu)} + \sqrt{\lambda} \|w\|_{L_{\vec{p}, q}(\bR^{d+1}_+, \mu)} \leq N \Big[\|G\|_{L_{\vec{p}, q}(\bR^{d+1}_+, \mu)} + \|\tilde{f}\|_{L_{\vec{p}, q}(\bR^{d+1}_+, \mu)} \Big]
\]
for $\lambda \geq \lambda_0 R_0^{-2}$. From this and by the triangle inequality, we obtain
\[
\begin{split}
& \|Du\|_{L_{\vec{p}, q}(\bR^{d+1}_+, \mu)} + \sqrt{\lambda} \|u\|_{L_{\vec{p}, q}(\bR^{d+1}_+, \mu)} \\
& \leq N\Big[ \|F\|_{L_{\vec{p}, q}(\bR^{d+1}_+, \mu)} + \|f\|_{L_{\vec{p}, q}(\bR^{d+1}_+, \mu)} \Big]\\
& \quad + N\Big[ \|Dv\|_{L_{\vec{p}, q}(\bR^{d+1}_+, \mu)} + \sqrt{\lambda} \|v\|_{L_{\vec{p}, q}(\bR^{d+1}_+, \mu)} \Big].
\end{split}
\]
The proof of Corollary \ref{thrm3} is completed.
\end{proof}
Next, we prove Corollary \ref{thrm2}
\begin{proof}[Proof of Corollary \ref{thrm2}] Since $\Omega$ is a bounded Lipschitz domain, it follows from the Stein extension theorem (see \cite[p. 181]{Stein}) that there is a linear, bounded extension map $\oE_0 :W^{1}_p(\Omega) \rightarrow W^1_p(\bR^d)$ satisfying
\[
\|\oE_0 (u)\|_{W^1_p(\bR^d)} \leq N(p, n, \Omega) \|u\|_{W^1_p(\Omega)}, \quad \forall \ u \in W^1_p(\Omega).
\]
As $\oE_0$ is linear, we see that $\oE_0(u) \in W^{1}_{p,q}(\bR^d\times (0,1), \mu)$ for all $u \in W^1_{p,q}(Q, \mu)$ and
\[
\|\oE_0 (u)\|_{W^1_{p,q}(\bR^d \times (0,1), \mu)} \leq N(p, n, \Omega) \|u\|_{W^1_{p, q}(Q, \mu)}, \quad \forall~u \in W^1_{p, q}(Q, \mu).
\]
Next, let $\phi_0 \in C_c^\infty(\bR)$ be the standard cut-off function such that $\phi_0 (y) =0$ for $y \geq 1$ and $\phi_0(y) = 1$ for $y \in [0, 1/2]$. We then define $\mathsf{T}_\Omega: W^1_{p,q}(Q, \mu) \rightarrow B^{\ell}_{p,q}(\Omega)$ as 
$$\mathsf{T}_\Omega (u)  = \mathsf{T} \circ \oE_0 (\bar{u})_{|\Omega}, $$
for $u \in W^1_{p,q}(Q, \mu)$, where $\bar{u}(x,y) = u(x,y) \phi_0(y), \ x \in \Omega, y \in \bR_+$, and $\mathsf{T}$ is defined as in Theorem \ref{thrm1}. It is then clear that  the assertion in Corollary \ref{thrm2} follows.
\end{proof}
\smallskip
\noindent
{\bf Acknowledgment}. T. Phan is partially supported by the Simons Foundation, grant \# 354889.

\end{document}